\documentclass[a4paper,10pt,reqno]{amsart}

\usepackage{amsthm,amsmath,amsfonts,amssymb}
\usepackage{pgf, tikz-cd}
\usetikzlibrary{arrows}

\usepackage[colorlinks=true,urlcolor=blue,
citecolor=red,linkcolor=blue,linktocpage,pdfpagelabels,
bookmarksnumbered,bookmarksopen]{hyperref}

\allowdisplaybreaks

\def\step#1#2{\par\noindent{\underline{\it Step~#1.}}\emph{ #2}\\}

\def\XXint#1#2#3{{\setbox0=\hbox{$#1{#2#3}{\int}$} \vcenter{\vspace{-1pt}\hbox{$#2#3$}}\kern-.5\wd0}}
\def\Xint#1{\mathchoice {\XXint\displaystyle\textstyle{#1}}{\XXint\textstyle\scriptstyle{#1}}{\XXint\scriptstyle\scriptscriptstyle{#1}}{\XXint\scriptscriptstyle\scriptscriptstyle{#1}}\!\int}
\def\intmed{\hbox{\ }\Xint{\hbox{\vrule height -0pt width 9pt depth 0.7pt}}}

\numberwithin{equation}{section}

\def\eps{\varepsilon}
\def\Chi#1{\hbox{{\large $\chi$}{\Large $_{_{#1}}$}}}
\def\C{\mathcal{C}}
\def\R{\mathbb{R}}

\def\S{\mathbb{S}}
\def\H{\mathcal{H}}
\def\comp{\subset\subset}
\def\res{\mathop{\hbox{\vrule height 5.5pt width .5pt depth 1pt \vrule height -.5pt width 5.5pt depth 1pt}}\nolimits}
\def\proofof#1{\begin{proof}[Proof of #1]}
\def\bal{\begin{aligned}}
\def\eal{\end{aligned}}

\theoremstyle{plain}
\newtheorem{thm}{Theorem}[section]
\newtheorem*{thm*}{Theorem}
\newtheorem{thmmain}{Theorem}

\newtheorem{defin}[thm]{Definition}

\title[The {\normalsize $\eps-\eps^\beta$} property for a double density]{The $\eps-\eps^\beta$ property in the isoperimetric problem with double density, and the regularity of isoperimetric sets}

\author{Aldo Pratelli}
\address{Universit\`a di Pisa, Dipartimento di Matematica, Largo Bruno Pontecorvo 5, IT-56127 Pisa}
\email{aldo.pratelli@unipi.it}
\author{Giorgio Saracco}
\address{Universit\`a di Pavia, Dipartimento di Matematica, via Ferrata 5, IT-27100 Pavia}
\email{giorgio.saracco@unipv.it}

\thanks{Both authors are members of the INdAM institute and have been partly supported by the INdAM--GNAMPA 2019 project ``Problemi isoperimetrici in spazi Euclidei e non'' (n.~prot.~U-UFMBAZ-2019-000473 11-03-2019).}



\begin{document}

\begin{abstract}
We prove the validity of the $\eps-\eps^\beta$ property in the isoperimetric problem with double density, generalising the known properties for the case of single density. As a consequence, we derive regularity for isoperimetric sets.
\end{abstract}

 \hspace{-3cm}
 {
 \begin{minipage}[t]{0.6\linewidth}
 \begin{scriptsize}
 \vspace{-2cm}
 This is a pre-print of an article published in \emph{Adv. Nonlinear Stud.}. The final authenticated version is available online at \href{https://doi.org/10.1515/ans-2020-2074}{https://doi.org/10.1515/ans-2020-2074}
 \end{scriptsize}
\end{minipage} 
}

\maketitle

\section{Introduction}

For any $N\geq 2$, we consider the isoperimetric problem in $\R^N$ with double density. More precisely, two lower semi-continuous (l.s.c.) and locally summable functions $f:\R^N\to (0,+\infty)$ and $h:\R^N\times \S^{N-1}\to (0,+\infty)$ are given, and we measure the volume and perimeter of any Borel set $E\subseteq\R^N$ according to the formulae
\begin{align*}
|E| := \int_E f(x)\, dx\,, && P(E):= \int_{\partial^* E} h(x,\nu_E(x))\, d\H^{N-1}(x)\,,
\end{align*}
where for every set of locally finite perimeter $E$ we denote by $\partial^* E$ its reduced boundary and by $\nu_E(x)\in\S^{N-1}$ the outer normal at $x \in \partial^* E$, while $P(E)=+\infty$ whenever $E$ is not a set of locally finite perimeter (see Section~\ref{sec:prel} for the basic definitions of the theory of sets of finite perimeter). We will refer to the ``single density'' case when $h(x,\nu)=f(x)$ for every $x\in\R^N$, and $\nu\in\S^{N-1}$. Whenever using the standard $N$-dimensional Lebesgue measure (i.e., $f\equiv 1$) or Euclidean perimeter (i.e., $h\equiv 1$), we shall use the subscript ${\rm Eucl}$, i.e. $|\cdot|_{\rm Eucl}$ and $P_{\rm Eucl}(\cdot)$.\par

The isoperimetric problem with single density is a wide generalisation of the classical Euclidean isoperimetric problem, and it has been deeply studied in the last decades, we refer the interested reader to~\cite{BBCLT, CMV, CRoS, DFP, FPu, FMP11, Gil18, MP, RCBM} and the references therein. The case of double density is yet a further important generalisation, since many of the possible applications correspond to two different densities. The simplest example is given by Riemannian manifolds, which locally behave as $\R^N$ with double density, being $f$ the norm of the Riemannian metric, and $h$ its derivative. In general, in view of applications it is crucial not only that $f$ and $h$ may differ, but also that $h$ may depend both on the point (from now on, the \emph{spatial variable}), and on the direction of the tangent space at that point (from now on the ~\emph{angular variable}).\par

The isoperimetric problem consists in finding, if they exist, the sets $E$ of minimal perimeter among those of fixed volume. The three main questions one is interested in are existence, boundedness and regularity of isoperimetric sets. Existence in the single density case has been studied in several papers, some of which are those quoted above; for the case of double density it has been studied either for some rather specific choices of the weights (radial~\cite{ABCMP17, CJQW, DGHKPZu, DHHT}, monomial~\cite{AFu, ABCMPu, ABCMP19, ABCMP19a}, Gauss-like~\cite{BCM16, FMP11} or in some Carnot groups~\cite{FM16, FS19}), or under very general conditions on $f$ and $h$, see~\cite{FPSu, PS, Sar17}. Concerning boundedness and regularity, a quite wide answer has been given in the paper~\cite{CP} for the single density case. The purpose of the present paper is to generalise the results of that paper to the case of double density.\par

A fundamental tool to study the isoperimetric problem is the classical ``$\eps-\eps$'' regularity property of Almgren, which basically says that one can locally modify a set $E$ by changing its volume by a small (positive or negative) quantity $\eps$, while increasing the perimeter by at most a quantity $C|\eps|$. In the case of single density, the $\eps-\eps$ property is true for any set of locally finite perimeter as soon as the density is at least locally Lipschitz, but otherwise it is in general false. This is the main reason why many regularity results for isoperimetric sets require a Lipschitz regularity hypothesis on the density.\par

To extend the study to the case of less regular densities, it is convenient to weaken the $\eps-\eps$ property to the so-called $\eps-\eps^\beta$ property, introduced in~\cite{CP}, that we immediately recall.

\begin{defin}[The $\eps-\eps^\beta$ property]\label{defepsepsbeta}
Let $E$ be a set of locally finite perimeter and $\beta \in [0,1]$. We say that $E$ possesses the $\eps-\eps^\beta$ property (relative to the densities $f$ and $h$) if for any ball $B$ such that $\H^{N-1}(B\cap \partial^* E)>0$ there exist constants $C>0$ and $\bar \eps > 0$ such that for all $|\eps|< \bar \eps$, there exists a set $F$ such that
\begin{align}\label{epsepsbeta}
F\Delta E \comp B\,, && |F| - |E| = \eps\,, && P(F)-P(E) \leq C|\eps|^\beta\,.
\end{align}
\end{defin}

In the case of single density, in~\cite{CP} it was shown that whenever $f$ is H\"older continuous with some exponent $0\leq\alpha\leq 1$, then the $\eps-\eps^\beta$ property holds for every set of locally finite perimeter for some $\beta=\beta(N, \alpha)$. Moreover, if $f$ is locally bounded and continuous, the $\eps-\eps^\beta$ property holds with $\beta=(N-1)/N$ and with \emph{any positive constant $C$}. This means that for any ball $B$ the constant $C>0$ in Definition~\ref{defepsepsbeta} can be taken arbitrarily small, up to choosing $\bar\eps$ small enough. As a consequence of the $\eps-\eps^\beta$ property, boundedness and ${\rm C}^{1,\eta}$ regularity of isoperimetric sets for some $\eta=\eta(N,\alpha)$ was obtained. We shall show that the results of~\cite{CP} can be extended to the general case of a double density. More precisely, we have the following three results.
\begin{thmmain}[The $\eps-\eps^\beta$ property]\label{thm:e-ebeta}
Assume that $f$ and $h$ are locally bounded, that $h$ is locally $\alpha$-H\"older in the spatial variable for some $\alpha \in [0,1]$, and that $E\subseteq\R^N$ is a set of locally finite perimeter. Then, $E$ possesses the $\eps-\eps^\beta$ property, where $\beta$ is given by
\begin{equation}\label{defbeta}
\beta=\beta(N, \alpha) = \frac{\alpha+(N-1)(1-\alpha)}{\alpha + N(1-\alpha)}\,.
\end{equation}
If $\alpha=0$ (in which case locally $\alpha$-H\"older precisely means locally bounded) and $h$ is continuous in the spatial variable, then $E$ possesses the $\eps-\eps^{\frac{N-1}N}$ property for all constants $C>0$ (that is, given any ball $B$ then the constant $C$ of Definition~\ref{defepsepsbeta} can be taken arbitrarily small, up to choosing $\bar\eps$ small enough).
\end{thmmain}

Notice that the $\eps-\eps^\beta$ property only requires the regularity of $h$ in the spatial variable, while no regularity of $h$ in the angular variable or of $f$ is required --- except that both $f$ and $h$ have to be l.s.c. and locally $L^1$, which is required in general to set the problem.

\begin{thmmain}[Boundedness]\label{thm:boundedness}
Assume that there exists a constant $M>0$ such that
\begin{align}\label{eq:bounds}
\frac 1M \leq f(x) \leq M\,, && \frac 1M \leq h(x, \nu) \leq M&& \forall\,x\in\R^N,\,\nu\in\S^{N-1}\,,
\end{align}
and that $E\subseteq\R^N$ is an isoperimetric set for which the $\eps-\eps^{\frac{N-1}N}$ property holds with an arbitrarily small constant $C$. Then, $E$ is bounded.
\end{thmmain}

Notice that this result, paired with Theorem~\ref{thm:e-ebeta}, ensures the boundedness of isoperimetric sets in a wide generality, i.e., whenever $f$ and $h$ are bounded and are away from $0$ -- that is, (\ref{eq:bounds}) holds -- and $h$ is continuous in the spatial variable.

\begin{thmmain}[Regularity of isoperimetric sets]\label{thm:regularity}
Assume that $f$ and $h$ are locally bounded. Then, any isoperimetric set is porous (see Definition~\ref{defrpp}), and its reduced boundary coincides $\H^{N-1}$-a.e. with its topological boundary. In addition, if $h=h(x)$ is locally $\alpha$-H\"older for some $\alpha\in (0,1]$, then $\partial^*E$ is ${\rm C}^{1, \eta}$, where
\begin{equation}\label{defeta}
\eta = \eta(N, \alpha) = \frac{\alpha}{2N(1-\alpha)+2\alpha}\,.
\end{equation}
\end{thmmain}
This regularity result is not sharp. In particular, in the $2$-dimensional case a higher regularity is shown in~\cite{CP2} for the case of single density. Such an improved regularity in the case of a double density and as well in the case $h=h(x, \nu)$ will be addressed in the forthcoming paper~\cite{PSu}.\par

Notice that, since $f$ and $h$ are l.s.c. and positive, they are always locally away from zero. Thus, in Theorems~\ref{thm:e-ebeta} and~\ref{thm:regularity}, $f$ and $h$ are locally bounded and locally away from zero, while in Theorem~\ref{thm:boundedness} these requests are made globally. These assumptions are essentially sharp. They can be slightly relaxed to the case of ``essentially bounded'' or ``essentially $\alpha$-H\"older'' functions as defined in~\cite[Definitions~1.6 and~1.7]{CP}. Loosely speaking, this relaxation allows the densities to take the values $0$ and $+\infty$ in a locally finite fashion. The very same can also be done in the present case. We preferred not to specify it in the claims, since on the one hand this generalisation makes the claims much less clear at first sight, and on the other hand it is a trivial generalisation once the proof has been completed.

The plan of the paper is as follows. In Section~\ref{sec:sketch} we give a short sketch of the proof of Theorem~\ref{thm:e-ebeta}, to give an idea of how the construction works and to explain where does the constant $\beta$ in~(\ref{defbeta}) come from. In Section~\ref{sec:prel} we recall some basic definitions and properties of sets of finite perimeter. In Section~\ref{sec:e-ebeta} we shall prove Theorem~\ref{thm:e-ebeta}, which we exploit in Section~\ref{sec:applications} to prove Theorem~\ref{thm:boundedness} and Theorem~\ref{thm:regularity}.

\subsection{A quick sketch of the proof of Theorem~\ref{thm:e-ebeta}\label{sec:sketch}}

Among the three main theorems of this paper, the first one is by far the hardest, and its proof is quite technical. Nevertheless, the overall idea is quite simple, and we outline it in this short section. In particular, the meaning of the value of $\beta$ in~(\ref{defbeta}) will appear as natural. Let us consider a very simplified situation, namely, we assume that the set $E$ is smooth. Of course, the main difficulty of the real proof is exactly to make everything precise also when dealing with non-smooth parts of the boundary. 
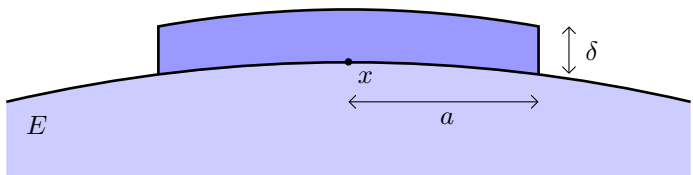
\begin{figure}[h!]
\begin{tikzpicture}
\draw[white] (-0.5,-1.2) rectangle (9.5,1.8);
\fill[blue!40!white] (2.0,0.37) -- (2.0,0.99) .. controls (3.66,1.29) and (5.33,1.29).. (7,0.99) -- (7,0.37);
\draw[line width=1] (2,0.37) -- (2,1) .. controls (3.66,1.3) and (5.33,1.3).. (7,1) -- (7,0.37);
\fill[blue!20!white] (0,-.01) .. controls (3,0.68) and (6,0.68) .. (9,-.01) -- (9,-1) -- (0,-1) -- cycle;
\draw[line width=1] (0,0) .. controls (3,0.7) and (6,0.7) .. (9,0);
\fill (4.5,0.53) circle (1.4pt);
\draw (0.4,-.3) node{$E$};
\draw (4.5,0.53) node[anchor=north west]{$x$};
\draw (5.8,-0.22) node{$a$};
\draw (7.7,.68) node{$\delta$};
\draw[angle 90-angle 90] (4.5,0) -- (7,0);
\draw[angle 90-angle 90] (7.4,0.37) -- (7.4,1);
\end{tikzpicture}

\caption{Sketch of the proof of Theorem~~\ref{thm:e-ebeta}.}\label{figsar}
\end{figure}

Let $x$ be a point of $\partial E$, and let us assume to fix the ideas that the outer normal vector to $\partial E$ at $x$ is vertical. Let $a$ be a very small quantity, and let us define $F$ by translating vertically of a lenght $\delta$ the part of $\partial E$ which is at distance less than $a$ from $x$, as depicted in Figure~\ref{figsar}. Since the difference between the volumes of $E$ and $F$ must be $\eps$, and since the density $f$ is locally bounded above by hypothesis and below by positivity and l.s.c., we obtain that
\begin{equation}\label{appbnd}
a^{N-1} \delta \approx \eps\,.
\end{equation}
Let us then evaluate the difference between the perimeters of $E$ and $F$. The boundary of $F$ coincides with the boundary of $E$ except for a ``vertical part'' (the two vertical segments in the figure) and for the fact that a ``horizontal'' piece of the boundary has been translated. The extra part consists of two segments of length $|\delta|$ in dimension $N=2$, while in general it is the lateral boundary of a cylinder of radius $a$ and height $|\delta|$, hence the perimeter behaves like $a^{N-2} |\delta|$. Finally, the part which is translated has $(N-1)$-dimensional measure of order $a^{N-1}$. Since it has been vertically translated of a length $\delta$, so in particular the normal vector remains the same, the $\alpha$-H\"older property of $h$ in the spatial variable ensures that the difference in perimeter is at most of order $a^{N-1} |\delta|^\alpha$. Hence, up to a multiplicative constant we can estimate
\[
P(F)-P(E) \leq a^{N-2}|\delta| + a^{N-1}|\delta|^\alpha\,.
\]
Optimising the choice of $a$ and $\delta$ subject to the constraint~(\ref{appbnd}), we select $a=|\eps|^\gamma$ with
\[
\gamma =\frac{1-\alpha}{\alpha+N(1-\alpha)}\,,
\]
from which the above estimate becomes precisely $P(F)-P(E)\leq |\eps|^\beta$ with $\beta$ given by~(\ref{defbeta}).

\subsection{Some properties of sets of locally finite perimeter\label{sec:prel}}

In this short section we recall some basic properties of sets of finite perimeter. A complete reference for the subject is for instance the book~\cite{AFP}, however the few things listed below make the present paper self-contained. In this section, whenever we write perimeter we mean the Euclidean one, while by volume the standard $N$-dimensional Lebesgue measure.\par

We say that a Borel set $E\subseteq\R^N$ is a \emph{set of locally finite perimeter} if its characteristic function $\Chi{E}$ is a $BV_{\rm loc}$ function. The \emph{reduced boundary} $\partial^* E$ of a Borel set $E\subseteq\R^N$ is the collection of points $x\in\R^N$ such that a direction $\nu(x)\in\S^{N-1}$ exists (and then it is necessarily unique) such that, calling
\[
B_r^\pm(x)= \big\{y\in\R^N,\, |y-x|< r,\, (y-x) \cdot \nu(x)\gtrless 0\big\}\,,
\]
one has
\begin{align*}
\lim_{r\searrow 0} \frac{|B_r^+(x)\cap E|_{\rm Eucl}}{|B_r^+(x)|_{\rm Eucl}} =0\,, &&
\lim_{r\searrow 0} \frac{|B_r^-(x)\cap E|_{\rm Eucl}}{|B_r^-(x)|_{\rm Eucl}} =1\,.
\end{align*}
We call the direction $\nu(x)$ the \emph{outer normal of $E$ at $x$}. A remarkable fact is that $E$ has locally finite perimeter if and only if $\partial^* E$ has locally finite $(N-1)$-dimensional Hausdorff measure $\H^{N-1}$. Moreover, for any set $E$ of locally finite perimeter one has
\[
\mu_E := D\Chi{E} = - \nu(x) \H^{N-1}\res \partial^* E\,.
\]
In particular, one defines the perimeter of any set by $P_{\rm Eucl}(E) = \H^{N-1}(\partial^* E)=|\mu_E|(\R^N)$, and a set has locally finite perimeter if and only if for any ball $B$, $E\cap B$ has finite perimeter. Obviously, whenever $E$ is regular enough, this notion of perimeter coincides with the classical one, and the reduced boundary $\partial^* E$ coincides with the topological boundary $\partial E$ up to $\H^{N-1}$-negligible subsets.\par
Another useful characterization of the boundary is the following. We say that a set $E\subseteq\R^N$ has \emph{density $d$ at a point $x\in\R^N$} if
\[
\lim_{r\searrow 0} \frac{|B_r^+(x)\cap E|_{\rm Eucl}}{|B_r^+(x)|_{\rm Eucl}} = d\,,
\]
and we call $E^d$ the collection of all the points with density $d$. By Lebesgue Theorem, $E=E^1$ and $\R^N\setminus E=E^0$ up to $\H^N$-negligible subsets, so that $\H^N\big(\R^N\setminus (E^0\cup E^1)\big)=0$. A stronger fact holds true, namely, a set $E$ has finite perimeter if and only if $\H^{N-1}\big(\R^N\setminus (E^0\cup E^1)\big)$ is finite, and moreover the three sets $\partial^* E,\, E^{1/2}$, and $\R^N\setminus(E^0\cup E^1)$ coincide up to $\H^{N-1}$-negligible subsets. Thus, the $(N-1)$-dimensional Hausdorff measure of each of these can be equivalently used as a generalised notion of perimeter.

Let us now conclude this section by listing two fundamental, well-known results about sets of locally finite perimeter. 

\begin{thm}[Blow-up]\label{thm:blowup}
Let $E\subseteq \R^N$ and let $x\in \partial^*E$. For every $\eps>0$, let $E_\eps := \frac 1\eps (E-x)$ be the blow-up set at $x$, write for the sake of brevity $\mu_\eps := \mu_{E_\eps}$ and call $H_x=\{y\in \R^N \,:\, y\cdot \nu(x)=0\}$ the tangent hyper-plane to $E$ at $x$, and $H_x^-=\{y\in \R^N \,:\, y\cdot \nu(x)<0\}$. Then, as $\eps \to 0^+$ the sets $E_\eps$ converge in the $L^1_{loc}$-sense to the half-space $H_x^-$, while the measures $\mu_\eps$ (resp. $|\mu_\eps|$) converge in the weak$^*$ sense to the measure $\nu_E(x) \H^{N-1} \res \partial H_x^-$ (resp. $ \H^{N-1} \res \partial H_x^-$).
\end{thm}

\begin{defin}[Vertical and horizontal sections]
Given a Borel set $E\subseteq \R^N$, we define its vertical section at level $y\in \R^{N-1}$, $E_y$, and its horizontal section at level $t\in \R$, $E^t$, respectively as
\begin{align*}
E_y := \{t\in \R\,:\, (y,t)\in E\}\,, && E^t := \{y\in \R^{N-1}\,:\, (y,t)\in E\}\,.
\end{align*}
\end{defin}

The following theorem, which goes by the name of \emph{Vol'pert Theorem} states that almost all (w.r.t. the proper dimensional Hausdorff measure) vertical sections and horizontal sections of sets of locally finite perimeter are of finite perimeter. Moreover, the reduced boundary of the sections coincides with the sections of the reduced boundary. The proof for vertical sections can be found in~\cite{AFP, Vol}, while for horizontal ones in~\cite{BCF, Fus, FMP}.

\begin{thm}[Vol'pert]\label{thm:Volpert}
Let $E$ be a set of locally finite perimeter. Then, for $\H^{N-1}$-a.e. $y\in \R^{N-1}$ the vertical section $E_y$ is a set of finite perimeter in $\R$, and $\partial^*(E_y) = (\partial^* E)_y$. Analogously for the horizontal sections, i.e. for $\H^1$-a.e. $t\in \R$ the horizontal section $E^t$ is a set of finite perimeter in $\R^{N-1}$, and $\partial^*(E^t) = (\partial^* E)^t$ up to an $\H^{n-2}$-negligible set.
\end{thm}

To fully understand the meaning of the Vol'pert Theorem, it is useful to consider what follows. The set $\partial^* E$ is a $(N-1)$-dimensional set, so for $\H^1$-almost every $t\in\R$ the section $(\partial^* E)^t$ is well defined up to $\H^{N-2}$-negligible subsets. Concerning $\partial^*( E^t)$, this is well-defined whenever the section $E^t$ is a well-defined $(N-1)$-dimensional set. However since $E$ has locally finite perimeter, $\H^{N-1}$-a.e. point of $\R^N$ has density either $1$, or $0$, or $1/2$. Moreover, the points of density $1/2$ have locally finite $\H^{N-1}$-measure. Consequently, only countably many sections $E^t$ carry a strictly positive $\H^{N-1}$-measure of points with density $1/2$. Hence, for a.e. $t\in\R$, $\H^{N-1}$-almost every point of $\R^N$, whose last coordinate equals $t$, has either density $0$ or $1$. For these $t$, the section $E^t$ is univocally defined, and thus the claim of Vol'pert Theorem makes perfect sense. A fully analogous consideration holds for vertical sections.

\section{Proof of the \texorpdfstring{$\eps-\eps^\beta$}{eps-eps to the beta} property}\label{sec:e-ebeta}

This whole section is devoted to the proof of the $\eps-\eps^\beta$ property. This is quite involved, thus for the sake of clarity we split it in several steps.

\begin{proof}[Proof of Theorem~\ref{thm:e-ebeta}]
Let $E\subseteq\R^N$ be a set of locally finite perimeter, and let $B\subseteq\R^N$ be a ball such that $\H^{N-1}(B\cap \partial^* E)>0$. Since $f$ and $h$ are locally bounded and l.s.c., there exists some constant $M>0$ such that
\begin{align*}
\frac 1M \leq f(y) \leq M\,, && \frac 1M \leq h(y, \nu) \leq M\,,&& \forall\,y\in B\,,\,\nu\in\S^{N-1}\,.
\end{align*}
Let $x\in\partial^* E\cap B$ be any point of the reduced boundary of $E$ in $B$. Up to a rotation and a translation, we may assume that $x=0$ and that $\nu_E(x)= (0,1)\in\R^{N-1}\times\R$.

\step{I}{Choice of the reference cube and the ``good'' part $G$.}
First of all, we let $\rho=\rho(M,N)>0$ be a small parameter, only depending on $M$ and on $N$, which will be chosen later on. As a direct application of the blow-up Theorem~\ref{thm:blowup} we obtain the existence of an arbitrarily small constant $a>0$ such that, calling $Q^N=(-a/2,a/2)^N$ and $Q=(-a/2,a/2)^{N-1}$ the $N$-dimensional and the $(N-1)$-dimensional cubes of side $a$, and letting $x=(x',x_N) \in \R^{N-1}\times \R$ one has
\begin{gather}
1-\rho \leq \frac{\H^{N-1}(\partial^* E \cap Q^N \cap \{-a\rho < x_N < a\rho\})}{a^{N-1}} \leq 1+\rho\,, \label{eq:1} \\
\H^{N-1}(\partial^* E \cap Q^N \setminus \{-a\rho < x_N < a\rho\}) \leq \rho \,a^{N-1}\,, \label{eq:2} \\
\H^N\Big( \big(E\Delta \{x_N<0\}\big)\cap Q^N\Big) < \rho^2 \,a^N\,,\label{eq:3bis}\\
\frac{\H^{N-2}\Big(\partial^* E \cap \partial Q^N\Big)}{2(N-1)a^{N-2}}\leq 1+\rho\,.\label{eq:4}
\end{gather}
In fact, the blow-up Theorem ensures the first three properties for every small $a$, and the validity of the last one for some arbitrarily small value of $a$ is then clear by integration. Notice that $a$ depends only on $\rho$ and on $E$, so ultimately $a=a(M, N, E)$. Observe that, again in view of the blow-up Theorem~\ref{thm:blowup}, inside the cube $Q^N$ one has to expect the set $E$ to be very close to the lower half-cube. As a consequence, a ``standard'' vertical section of $E$ inside $Q^N$ should be close to the segment $(-a/2,0)$. Writing then for brevity $E_{x'}^Q=E_{x'}\cap Q^N$ and $\partial^* E_{x'}^Q=\partial^* E_{x'}\cap Q$, we define then $G\subseteq Q$ the ``good'' sections, that is,
\[
G:= \Big\{x'\in Q:\, \partial^*(E_{x'})=(\partial^* E)_{x'},\ \#\big(\partial^* E_{x'}^Q\big)=1,\ \partial^*E_{x'}^Q\subseteq (-a\rho,a\rho),\ E_{x'}^Q\subseteq (-a/2, a\rho)\Big\}\,.
\]
We want to prove that the set $G$ covers most of the cube $Q$, and actually only very few perimeter is carried by sections which are not in $G$. More precisely, we will prove that
\begin{gather}
\H^{N-1}(Q \setminus G) \leq 4\rho a^{N-1}\,, \label{GtuttoQ}\\
\H^{N-1}\Big(\partial^* E \cap \big( (Q\setminus G )\times (-a/2,a/2) \big)\Big)\leq 6\rho a^{N-1}\label{perGtuttoQ}\,.
\end{gather}
To obtain these estimates, let us consider $x'\in Q$ such that $x'\notin G$. This can happen for different reasons, for instance $x'$ can belong to the set $\Gamma_1$ of the sections for which $\partial^*(E_{x'})\neq (\partial^* E)_{x'}$, or to the set $\Gamma_2$ of the sections such that $(\partial^* E)_{x'}$ contains at least two points. Let us then call $\Gamma_3=Q\setminus (G\cup \Gamma_1\cup\Gamma_2)$, and consider a point $x'$ in $\Gamma_3$. We have that $\partial^*(E_{x'})=(\partial^* E)_{x'}$ and that it contains either no point, or a single point. In this second case, this single point either does not belong to $(-a\rho,a\rho)$, or it belongs to $(-a\rho,a\rho)$ and the section $E_{x'}$ is the upper part of the segment instead of the lower one, that is, $E_{x'}\subseteq (-a\rho,a/2)$. In any case, we immediately obtain that for any $x'\in \Gamma_3$ one has
\[
\H^1\big(E_{x'}\Delta (-a/2,0)\big) \geq a\rho\,,
\]
and by~(\ref{eq:3bis}) we deduce
\begin{equation}\label{estiG3}
\H^{N-1}(\Gamma_3)\leq \rho\,a^{N-1}\,.
\end{equation}
Instead, Vol'pert Theorem~\ref{thm:Volpert} gives that
\begin{equation}\label{estiG1}
\H^{N-1}(\Gamma_1)=0\,.
\end{equation}
Since the projection on the first $N-1$ coordinates is $1$-Lipschitz and by definition of $\Gamma_2$ we have
\begin{align}
\H^{N-1}&\Big(\partial^* E \cap \big(G\times (-a/2,a/2)\big)\Big) \geq \H^{N-1} (G)\,,\label{PG>G}\\
\H^{N-1}&\Big(\partial^* E \cap \big(\Gamma_2\times (-a/2,a/2)\big)\Big) \geq 2 \H^{N-1} (\Gamma_2)\,.\nonumber
\end{align}
Therefore, on the one hand by~(\ref{estiG1}) and~(\ref{estiG3}) we get
\[\begin{split}
\H^{N-1}(\partial^* E \cap Q^N) &\geq \H^{N-1}(G)+ 2 \H^{N-1}(\Gamma_2)
= a^{N-1} - \H^{N-1}(\Gamma_1)-\H^{N-1}(\Gamma_3)+ \H^{N-1}(\Gamma_2)\\
&\geq (1-\rho) a^{N-1} + \H^{N-1}(\Gamma_2)\,.
\end{split}\]
On the other hand, (\ref{eq:1}) and~(\ref{eq:2}) give
\begin{equation}\label{PerinQ}
\H^{N-1}(\partial^* E \cap Q^N)\leq (1+2\rho)a^{N-1}\,,
\end{equation}
so we deduce
\[
\H^{N-1}(\Gamma_2)\leq 3\rho\,a^{N-1}\,.
\]
Since $Q=G\cup (\Gamma_1\cup\Gamma_2\cup\Gamma_3)$, this estimate together with~(\ref{estiG1}) and~(\ref{estiG3}) implies~(\ref{GtuttoQ}). Finally, (\ref{GtuttoQ}) together with~(\ref{PG>G}) and~(\ref{PerinQ}) directly gives~(\ref{perGtuttoQ}).

\step{II}{An estimate about small cubes.}
In this step we prove a simple estimate about the perimeter on small cubes. More precisely, for every $x'\in\R^{N-1}$ and $\ell>0$ let us call $Q_\ell(x')\subseteq\R^{N-1}$ the cube with center in $x'$, side $\ell$, and with sides parallel to the coordinate planes. Let us then fix some $c\in Q$ and some $\ell>0$ such that $Q_{2\ell}(c)\comp Q$. We aim to show that
\begin{equation}\label{fattobene}\begin{split}
\int_{Q_\ell(c)} \H^{N-2} &\Big(\partial^* E \cap \big(\partial Q_\ell(x')\times (-a/2,a/2)\big)\Big) \, d\H^{N-1}(x')\\
&\leq (N-1)\ell^{N-2} \H^{N-1}\Big(\partial^* E \cap\big( Q_{2\ell}(c)\times (-a/2,a/2)\big)\Big)\,.
\end{split}\end{equation}
To prove this estimate, let us fix a direction $1\leq i \leq N-1$, and for every $x'\in Q_\ell(c)$ let us call
\[
S_i^-(x')=\big\{y\in Q_{2\ell}(c):\, y_i=x_i-\ell\big\}\,.
\]
We have then
\[\begin{split}
\int_{Q_\ell(c)} \H^{N-2}&\Big(\partial^* E \cap \big(S_i^-(x')\times (-a/2,a/2)\big)\Big)\,d\H^{N-1}(x')\\
&=\ell^{N-2} \int_{-3\ell/2}^{-\ell/2} \H^{N-2}\Big(\partial^* E \cap \big(Q_{2\ell}(c)\times (-a/2,a/2)\big)\cap \{x_i = c_i+t\}\Big) \, dt\\
&\leq \ell^{N-2}\H^{N-1}\Big(\partial^* E \cap \big(Q_{2\ell}(c)\times (-a/2,a/2)\big)\cap \{c_i-3/2\ell<x_i<c_i-\ell/2\}\Big)\,.
\end{split}\]
The same estimate is clearly valid replacing $S_i^-$ with $S_i^+$, defined in the same way with $y_i=x_i+\ell$ in place of $y_i=x_i-\ell$. Since for every $x'\in Q_\ell(c)$ one clearly has $\partial Q_\ell(x')\subseteq \cup_{i=1}^{N-1} S_i^-(x')\cup S_i^+(x')$, summing the last estimate among all $1\leq i \leq N-1$, we immediately get~(\ref{fattobene}).

\step{III}{Selection of ``good'' horizontal cubes $Q_j$.}
We now fix $0\leq \gamma<\frac 1{N-1}$ with $\gamma=\gamma(N, \alpha)$, whose precise value is given in~(\ref{tobecited}). Moreover, we fix a very small constant $\bar\eps>0$, which will be precised later, and which depends on $a$, $M$ and $\gamma$, so that in the end $\bar\eps = \bar \eps(M, N, E, \alpha)$. Let also $0<\eps<\bar\eps$ be given. In this step, we define $H=H(\gamma,N)$ pairwise disjoint cubes $Q_j\subseteq Q$, and in the next step we will select one of them to go on with the construction. If $\gamma=0$, we only take one cube, that is $H(0,N)=1$ and the unique cube is $Q_1=Q$ itself. If $\gamma>0$, instead, we take pairwise disjoint $(N-1)$-dimensonal open cubes $Q_{2\ell}(x_j^+)\comp Q$, where we write for brevity $\ell=a\eps^\gamma$. Notice that it is possible to find $2H$ of these pairwise disjoint cubes with
\begin{equation}\label{estiH}
H\geq \frac 1{2^{N+1} \eps^{\gamma(N-1)}}\,.
\end{equation}
For each cube, we can find a point $x_j\in Q_\ell(x_j^+)$ such that, calling $Q_j=Q_\ell(x_j)$,
\[
\H^{N-2} \Big(\partial^* E \cap \big(\partial Q_j\times (-a/2,a/2)\big)\Big)
\leq \intmed_{Q_\ell(x_j^+)} \H^{N-2} \Big(\partial^* E \cap \big(\partial Q_\ell(x')\times (-a/2,a/2)\big)\Big) \, d\H^{N-1}\,,
\]
which by~(\ref{fattobene}) gives
\[
\H^{N-2} \Big(\partial^* E \cap \big(\partial Q_j\times (-a/2,a/2)\big)\Big) \leq \frac {N-1}{a\eps^\gamma}\,\H^{N-1}\Big(\partial^* E \cap\big( Q_{2\ell}(x_j^+)\times (-a/2,a/2)\big)\Big)\,.
\]
Keeping in mind that the cubes $Q_{2\ell}(x_j^+)$ are pairwise disjoint and contained in $Q$, as well as~(\ref{PerinQ}), we deduce that for all $1\leq j\leq H$
\begin{equation}\label{sectN-2small}
\H^{N-2} \Big(\partial^* E \cap \big(\partial Q_j\times (-a/2,a/2)\big)\Big) \leq N 2^{N+1} (a\eps^\gamma)^{N-2}\,,
\end{equation}
provided that $\rho\ll 1/N$. In addition, by Vol'pert Theorem~\ref{thm:Volpert}, we can also assume that for every $1\leq j\leq H$
\begin{align}\label{valvol}
\partial^*E \cap \big(\partial Q_j \times (-a/2, a/2)\big) = \partial^*\Big( E \cap \big(\partial Q_j \times (-a/2, a/2)\big)\Big) && \H^{N-2}\text{-a.e.}\,.
\end{align}
Notice that $E \cap \big(\partial Q_j \times (-a/2, a/2)\big)$ is contained in the union of $2(N-1)$ hyperplanes of dimension $N-1$, so the boundary in the right-hand side above has to be intended as the corresponding union of the $(N-2)$-dimensional boundaries. Observe that also in the case $\gamma=0$ the unique cube $Q_1=Q$ satisfies~(\ref{sectN-2small}), which is in fact weaker than~(\ref{eq:4}), and without loss of generality we can also assume the validity of~(\ref{valvol}) by Vol'pert Theorem.

\step{IV}{Choice of one of the horizontal cubes $Q_j$.}
In this step we select a particular horizontal cube among those defined in the previous step. In particular, we aim to find some $1\leq j\leq H$ such that the cubes $Q_\eps=Q_j$ and $Q^N_\eps=Q_j\times (-a/2,a/2)$ satisfy
\begin{gather}
\frac{\H^{N-1}(\partial^*E \cap Q^N_\eps)}{(a\eps^\gamma)^{N-1}} \leq 1+2^{N+5}\rho\,,\label{eq:10}\\
\frac{\H^{N-1}(\partial^*E \cap Q_\eps^N \setminus \{-a\rho < x_N < a\rho\})}{(a\eps^\gamma)^{N-1}} \leq 2^{N+3} \rho\,,\label{eq:11}\\
\frac{\H^N(E\cap Q_\eps^N \cap \{x_N>0\})}{a^N \eps^{\gamma(N-1)}} \leq 2^{N+3}\rho^2\,,\label{eq:12}\\
\frac{\H^{N-1}(Q_\eps\setminus G)}{(a\eps^\gamma)^{N-1}} \leq 2^{N+5}\rho\,,\label{eq:13}
\end{gather}
being $G$ the set defined in Step~I. Let us start noticing that everything is trivial for the special case $\gamma=0$. Indeed, in this case there is nothing to choose since there is only one cube $Q_1=Q$, hence we only have to check the validity of the properties~(\ref{eq:10})--(\ref{eq:13}). And in turn, the estimates~(\ref{eq:1}), (\ref{eq:2}), (\ref{eq:3bis}) and~(\ref{GtuttoQ}) exactly provide the validity of these four inequalities, with even better constants. By the way, in this particular case the cubes $Q_\eps=Q$ and $Q_\eps^N=Q^N$ actually do not even depend on $\eps$.\par

We focus then on the non-trivial case $\gamma>0$. First of all, keeping in mind~(\ref{eq:2}) and that the cubes $Q_j$ are pairwise disjoint we obtain that strictly less than $1/4$ of the $H$ cubes $Q_j$ are those for which
\[
\H^{N-1}(\partial^*E \cap Q_\eps^N \setminus \{-a\rho < x_N < a\rho\}) > \frac{4\rho a^{N-1}}H\,,
\]
hence by~(\ref{estiH}) for more than $3/4$ of the cubes~(\ref{eq:11}) is valid. With the very same argument one obtains, for more than $3/4$ of the cubes $Q_j$, also the validity of~(\ref{eq:12}) from~(\ref{eq:3bis}) and of~(\ref{eq:13}) from~(\ref{GtuttoQ}).\par

The argument to obtain~(\ref{eq:10}) is slightly more involved. More precisely, calling $A$ the projection over $Q$ of $\partial^* E\cap Q^N$, for every Borel set $V\subseteq Q$ we define
\[
\zeta(V) = \H^{N-1}\big(\partial^* E \cap (V\times (-a/2,a/2)\big) - \H^{N-1}(V\cap A)\,.
\]
It is immediate to notice that $\zeta$ is a positive measure and, observing that $Q\setminus A\subseteq \Gamma_3$, from~(\ref{estiG3}) and~(\ref{PerinQ}) we get $\zeta(Q)\leq 3\rho a^{N-1}$. Hence, arguing as for the other inequalities, we obtain that for more than $3/4$ of the $H$ cubes $Q_j$ one has
\[
\zeta(Q_j) \leq \frac{12\rho a^{N-1}}H\leq 2^{N+5} \rho (a\eps^\gamma)^{N-1}\,,
\]
which implies~(\ref{eq:10}). Summarising, we can find at least one index $1\leq j\leq H$ such that all the estimates~(\ref{eq:10})--(\ref{eq:13}) hold true for such an index. For future use we remark that, putting together~(\ref{eq:10}) and~(\ref{eq:13}), we have
\begin{equation}\label{futref}
\H^{N-1}\Big(\partial^* E \cap \big((Q_\eps\setminus G)\times (-a/2,a/2)\big) \Big)
\leq 2^{N+6}\rho (a\eps^\gamma)^{N-1}\,.
\end{equation}

\step{V}{Definition of $\sigma^\pm$ and of the modified sets $F_\delta\subseteq\R^N$.}
In this step we define a one-parameter family of sets $F_\delta$; in the next step we will select a suitable $\delta$ such that $F=F_\delta$ satisfies the volume constraint in~(\ref{epsepsbeta}), and then we will check that this set $F$ also satisfies the perimeter constraint in~(\ref{epsepsbeta}). To present the construction, we need yet another constant $\bar\delta\ll a$, which depends on $a,\,M,\,\gamma$ and $\eps$, so ultimately $\bar \delta = \bar\delta(M, N, E,\alpha, \eps)$. The precise value of $\bar\delta$ is given in~(\ref{precisebardelta}). We define the integer
\[
K:= \left \lfloor \frac a{3\bar \delta} \right \rfloor\,,
\]
and we take $K$ pairwise disjoint open strips $S_i = Q_\eps \times (\sigma_i,\, \sigma_i + \bar\delta) \subseteq Q_\eps^N$, with constants $a\rho<\sigma_i<\frac a2 - \bar\delta$ for $1\leq i \leq K$. Notice that this is possible in view of the definition of $K$, in particular we can do this in such a way that also the closures $\overline{S_i}$ of the strips are pairwise disjoint. Then, by~\eqref{eq:11} and~\eqref{eq:12}
\[\begin{split}
\sum_{i=1}^K a\H^{N-1}(\partial^*E \cap \overline{S_i}) + \H^N(E\cap S_i) & \leq a\H^{N-1}(\partial^*E \cap (Q_\eps \times (a\rho, a/2)) + \H^N(E\cap (Q_\eps \times (a\rho, a/2))\\
&\leq 2^{N+3}\rho a^N\eps^{\gamma(N-1)} + 2^{N+3}\rho^2 a^N\eps^{\gamma(N-1)} \leq 2^{N+4}\rho a^N \eps^{\gamma(N-1)}\,.
\end{split}\]
Then, there exists at least one strip $S_+:=S_{\bar i}$, such that
\begin{equation}\label{eq:sigma+}
a\H^{N-1}(\partial^*E \cap \overline{S_+}) + \H^N(E\cap S_+) \leq \frac{2^{N+4}\rho a^N \eps^{\gamma(N-1)}}{K} \leq 2^{N+6}\bar \delta \rho (a\eps^\gamma)^{N-1}\,.
\end{equation}
For the sake of notation, we write $\sigma^+ = \sigma_{\bar i}$, so that $S_+ = Q_\eps \times (\sigma^+, \sigma^+ +\bar \delta)$ and we keep in mind that $a\rho<\sigma^+<\frac a2-\bar\delta$.\par

On the other hand, by~\eqref{eq:11} we can estimate
\begin{align*}
2^{N+3}\rho &\geq \frac{\H^{N-1}(\partial^*E \cap Q_\eps^N \cap \{ x_N <-a\rho\})}{(a\eps^\gamma)^{N-1}}
\geq \int_{-\frac a2}^{-a\rho} \frac{\H^{N-2}((\partial^*E \cap Q_\eps^N)^t)}{(a\eps^\gamma)^{N-1}}\,dt\\
&=\left(\frac 12 -\rho \right) \intmed_{-\frac a2}^{-a\rho} \frac{\H^{N-2}\big(\partial^* (E^t) \cap Q_\eps \big)}{a^{N-2}\eps^{\gamma(N-1)}}\,dt \,,
\end{align*}
where the last equality follows from Vol'pert Theorem. In particular, for almost every $-a/2<t<-a\rho$ one has that $\H^{N-1}$-a.e. point of the hyper-plane $\R^{N-1}\times \{t\}$ has density either $0$ or $1$, and
\[
\big(\partial^* E\cap Q_\eps^N\big)^t=\partial^*\big(E^t \cap Q_\eps\big)\,.
\]
Consequently, we can find another section $-\frac a2<\sigma^-<-a\rho$ for which in addition
\begin{equation}\label{eq:sigma-}
\H^{N-2}\big(\partial^* (E^{\sigma^-}) \cap Q_\eps \big)
\leq 3\cdot 2^{N+3}\rho a^{N-2}\eps^{\gamma (N-1)}\,.
\end{equation}
For every $0<\delta\leq \bar\delta$ we define the set $F_\delta$ as
\[
F_\delta:= E \setminus \big( Q_\eps\times (\sigma^-,\sigma^++\delta)\big) \cup \big(\big(E^{\sigma^-}\!\!\cap Q_\eps\big)\times (\sigma^-,\sigma^-+\delta)\big)\cup \big\{(x',x_N+\delta):\, (x',x_N)\in E \cap (Q_\eps\times (\sigma^-,\sigma^+))\big\}\,.
\]
It is easy but important to observe how the set $F_\delta$ is defined. The difference between $E$ and $F_\delta$ is only inside the cylinder $Q_\eps\times (\sigma^-,\sigma^++\delta)$. In this cylinder, the flat basis $E^{\sigma^-}\cap Q_\eps$ is stretched vertically of a height $\delta$, the high central part $E\cap (Q_\eps\times (\sigma^-,\sigma^+))$ is translated vertically of $\delta$, and the short upper part $E\cap (Q_\eps\times (\sigma^+,\sigma^++\delta))\subseteq E\cap S_+$ is simply removed. Keep in mind that $E$ has density either $0$ or $1$ at $\H^{N-1}$-a.e. point of the basis $Q_\eps\times \{\sigma^-\}$, therefore the vertical stretch of $E^{\sigma^-}\cap Q_\eps$ makes sense.

\step{VI}{Volume evaluation and choice of the competitor $F$.}
In this step we estimate the volume of the sets $F_\delta$ defined in the previous step. We will then select a particular one of these sets, which we will simply denote by $F$, such that
\begin{equation}\label{volumeF}
|F| = |E|+\eps\,.
\end{equation}
We shall then show that such a set is the one required by the $\eps-\eps^\beta$ property, by checking the validity of the inequality $P(F)\leq P(E) + C\eps^\beta$.\par

For a given $0<\delta<\bar\delta$, let us define
\begin{align*}
E^+ = F\setminus E\,, && E^-_1 = (E\setminus F) \cap (Q_\eps\times (\sigma^-,\sigma^+))\,, && E^-_2 = (E\setminus F) \cap (Q_\eps\times (\sigma^+,\sigma^++\delta))\,,
\end{align*}
so that
\begin{equation}\label{setting}
|F| = |E| + |E^+| - |E^-_1|- |E^-_2|\,.
\end{equation}
The last term is the easiest to estimate. Indeed, since of course $E^-_2 \subseteq E \cap (Q_\eps\times (\sigma^+,\sigma^++\delta))\subseteq E\cap S_+$, recalling that $\frac 1 M <f,\, h < M$ in $Q^N$, by~(\ref{eq:sigma+}) we directly have
\begin{equation}\label{sett1}
|E^-_2| \leq |E\cap S_+| \leq M \H^N (E\cap S_+)
\leq 2^{N+6}M \bar \delta \rho (a\eps^\gamma)^{N-1}\,.
\end{equation}
Let us now pass to estimate $|E^+|$ and $|E^-_1|$. To do so, it is convenient to consider separately the vertical sections corresponding to some $x'\in G$ and the other ones. We start by picking some $x'\in Q_\eps \cap G$. Then, $(E\cap Q_\eps)_{x'}$ is a vertical segment of the form $(-a/2,y)$ with some $-a\rho < y <a \rho$, so in particular $\sigma^- < y <\sigma^+$. Therefore, the section $(E^-_1)_{x'}$ is empty, while the section $(E^+)_{x'}$ is the segment $(y,y+\delta)$. By Fubini and~(\ref{eq:13}), we have then
\begin{equation}\label{sett2}
\begin{array}{c}
\big|E^-_1 \cap (G\times (-a/2,a/2))\big|=0\,,\\
\bal\frac{\delta(a\eps^\gamma)^{N-1}}M\, (1-2^{N+5}\rho) \leq\big|E^+ \cap (G\times (-a/2,a/2))\big|\leq M\delta (a\eps^\gamma)^{N-1}\eal \,.
\end{array}
\end{equation}
Let us now instead take $x'\in Q_\eps\setminus G$, and assume that $(x',x_N)\in E^+$ for some $x_N\in (\sigma_-,\sigma^++\delta)$. The fact that $(x',x_N)\in E^+=F\setminus E$ implies that $(x',x_N)\notin E$. Instead, the fact that $(x',x_N)\in F$ implies that $(x',y)\in E$ where $y=\max \{x_N-\delta, \sigma^-\}$. As a consequence, the segment $\{x\}\times (x_N-\delta,x_N)$ intersects $\partial^* (E_{x'})$. Analogously, if $(x',x_N)\in E^-_1$, this means that $(x',x_N)\in E$ while $(x',\max\{x_N-\delta,\sigma^-\})\notin E$, so that $\{x\}\times (x_N-\delta,x_N)$ intersects again $\partial^* (E_{x'})$. Keeping in mind Vol'pert Theorem, we deduce that
\[
(E^+ \cup E^-_1) \cap \Big((Q_\eps\setminus G)\times (-a/2,a/2) \Big)\subseteq \tau\Big(\partial^* E\cap \big(Q_\eps\setminus G\times (-a/2,a/2)\big)\Big)\,,
\]
where for every set $A\subseteq\R^N$ we define
\begin{equation}\label{hereistau}
\tau(A) = \big\{(x',x_N+t\delta)\,, (x',x_N)\in A,\, 0\leq t\leq 1\big\}\,.
\end{equation}
Therefore, also by~(\ref{futref}) we deduce
\begin{equation}\label{sett3}\begin{split}
\big| (E^+ \cup E^-_1) &\cap \big((Q_\eps\setminus G)\times (-a/2,a/2) \big)\big|
\leq M \H^N\Big((E^+ \cup E^-_1) \cap \big((Q_\eps\setminus G)\times (-a/2,a/2)\big)\Big)\\
&\leq M \delta \H^{N-1}\Big(\partial^* E\cap \big(Q_\eps\setminus G\times (-a/2,a/2)\big)\Big)
\leq 2^{N+6} M \delta \rho (a\eps^\gamma)^{N-1}\,.
\end{split}\end{equation}
We are finally in a position to define $\bar\delta$ as
\begin{equation}\label{precisebardelta}
\bar \delta = \frac{2M\eps}{(a\eps^\gamma)^{N-1}}\,.
\end{equation}
Keep in mind that our construction makes sense only if $\bar\delta\ll a$, which in turn is true as soon as we have chosen $\bar\eps$ small enough (recall that $\bar\eps$ could be chosen depending on $a,\, M$ and $\gamma$, and that $0\leq \gamma<\frac 1{N-1}$). Putting together~(\ref{setting}), (\ref{sett1}), (\ref{sett2}) and~(\ref{sett3}), and recalling~(\ref{precisebardelta}), in the case $\delta=\bar\delta$ we can now estimate
\begin{equation}\label{toobig}
|F_{\bar\delta}| - |E| \geq
\bar\delta (a\eps^\gamma)^{N-1}\Big(\frac {1- 2^{N+5}\rho}M - 2^{N+7}M \rho\Big)
=2M\eps \Big(\frac {1- 2^{N+5}\rho}M - 2^{N+7}M \rho\Big) > \eps\,,
\end{equation}
where the last inequality is true up to taking $\rho$ small enough depending on $M$ and $N$ (keep in mind that $\rho$ was chosen precisely depending on $M$ and on $N$). On the other hand, in the case $\delta=\bar\delta/(4M^2)$, (\ref{setting}), (\ref{sett2}) and~(\ref{sett3}), together with~(\ref{precisebardelta}), allow to deduce
\begin{equation}\label{toosmall}
|F_{\bar\delta/(4M^2)}| - |E| \leq \frac{\bar\delta}{4M^2} \,(a\eps^\gamma)^{N-1}\big(M +2^{N+6} M \rho\big)
=\frac\eps{2M} \big(M +2^{N+6} M \rho\big)<\eps\,,
\end{equation}
again up to choosing $\rho$ small enough. Since of course the measure of $F_\delta$ is continuous on $\delta$, by~(\ref{toobig}) and~(\ref{toosmall}) we deduce the existence of some $\delta$ between $\bar\delta/(4M^2)$ and $\bar\delta$ such that~(\ref{volumeF}) holds. We fix then this $\delta$, and from now on we let $F=F_\delta$.

\step{VII}{Evaluation of $P(F)$.}
In this step we evaluate $P(F)$. Notice that, in order to obtain the $\eps-\eps^\beta$ property, we need to show that $P(F)\leq P(E)+C\eps^\beta$. Since $F$ equals $E$ outside of the open cylinder $\C=Q_\eps\times (\sigma^-,\sigma^++\delta)$, the difference between $\partial^* F$ and $\partial^* E$ is contained in the closed cylinder. We split $\partial^* F \cap \overline\C$ in four parts, namely $T^+_{\rm top}$, $T^+_{\rm btm}$, $T^+_{\rm ltr}$ and $T^+_{\rm int}$, where
\begin{equation}\label{ult1}\begin{aligned}
T^+_{\rm top}&=\partial^* F \cap \big(Q_\eps \times \{\sigma^++\delta\}\big)\,, &
T^+_{\rm btm}&=\partial^* F \cap \big(Q_\eps \times \{\sigma^-\}\big)\,, \\
T^+_{\rm ltr}&=\partial^* F \cap \big(\partial Q_\eps \times (\sigma^-,\sigma^++\delta)\big)\,,\qquad &
T^+_{\rm int}&=\partial^* F \cap \C\,.
\end{aligned}\end{equation}
We also split $\partial^* E\cap \overline \C$ in four parts in a slightly different way as
\begin{equation}\label{ult2}\begin{aligned}
T^-_{\rm top}&=\partial^* E \cap \big(Q_\eps \times [\sigma^+,\sigma^++\delta]\big)\,, &
T^-_{\rm btm}&=\partial^* E \cap \big(Q_\eps \times \{\sigma^-\}\big)\,, \\
T^-_{\rm ltr}&=\partial^* E\cap \big(\partial Q_\eps \times (\sigma^-,\sigma^++\delta)\big)\,,\qquad &
T^-_{\rm int}&=\partial^* E \cap \big(Q_\eps \times (\sigma^-,\sigma^+)\big)\,.
\end{aligned}\end{equation}
We now have to carefully consider the above pieces. The easiest thing to notice is that, since by definition of $\sigma^-$ the set $E$ has density either $0$ or $1$ at $\H^{N-1}$-almost every point of $Q_\eps\times \{\sigma^-\}$, then
\begin{equation}\label{piece1}
\H^{N-1}\big(T^+_{\rm btm} \cup T^-_{\rm btm}\big) =0\,,
\end{equation}
so neither $E$ nor $F$ carry any perimeter on the bottom face of the cylinder $\C$.\par
Let us now consider $T^\pm_{\rm top}$. We aim to show that, up to $\H^{N-1}$-negligible subsets,
\begin{equation}\label{topinclusion}
T^+_{\rm top} \subseteq \pi\big(T^-_{\rm top}\big)\,,
\end{equation}
where $\pi(x',x_N)=(x',\sigma^++\delta)$ denotes the projection over the hyperplane $\{x_N=\sigma^++\delta\}$. By the properties of sets of finite perimeter, for $\H^{N-1}$-almost every $x'\in Q_\eps$ we have that $E$ has density either $0$, or $1/2$, or $1$ both at $(x',\sigma^+)$ and at $(x',\sigma^++\delta)$. If at least one of them is $1/2$, then $x'$ belongs to the right set in~(\ref{topinclusion}). If they are both $0$ or both $1$, then by construction $F$ has also density $0$ or $1$ at the point $(x',\sigma^++\delta)$, which then does not belong to $T^+_{\rm top}$. Let us then assume that one of the two densities is $0$ and the other one is $1$. Up to discarding an $\H^{N-1}$-negligible quantity of points $x'$, we can assume that Vol'pert Theorem holds for the section $E_{x'}$, and the density of $E_{x'}$ at the points $\sigma^+$ and $\sigma^++\delta$ is once $0$ and once $1$. Consequently, $\partial^*(E_{x'})$ contains some point in the segment $(\sigma^+,\sigma^++\delta)$, and then again $x'$ belongs to the right set in~(\ref{topinclusion}). Summarising, we proved the inclusion~(\ref{topinclusion}) up to sets of zero $\H^{N-1}$-measure. Keeping in mind the fact that the projection is $1$-Lipschitz and~(\ref{eq:sigma+}), and noticing that $T^-_{\rm top}\subseteq \partial^* E\cap \overline{S_+}$, we have then
\begin{equation}\label{piece2}
\int_{T^+_{\rm top}} h(y,\nu_F(y))\,d\H^{N-1}(y) \leq M \H^{N-1}(T^+_{\rm top})
\leq M \H^{N-1}(T^-_{\rm top})
\leq 2^{N+6} M \bar \delta \rho a^{N-2} (\eps^\gamma)^{N-1}\,.
\end{equation}
To consider the set $T^+_{\rm ltr}$, we can argue similarly, keeping in mind that by~(\ref{valvol}) Vol'pert Theorem is true for $\H^{N-2}$-almost each $x'\in \partial Q_\eps$. Indeed, take $(x',x_N)$ where $x'$ is a point in $\partial Q_\eps$ at which Vol'pert Theorem holds true, and $\sigma^-+\delta<x_N<\sigma^++\delta$. Up to $\H^{N-1}$-negligible subsets, $E$ has density either $0$, or $1/2$ or $1$ both at $(x',x_N)$ and at $(x',x_N-\delta)$. If the densities are both $0$ or both $1$, then $(x',x_N)\notin \partial^* F$, while if one density is $0$ and the other one is $1$ then again there is some point in the segment $\{x'\}\times (x_N-\delta,x_N)$ which belongs to $\partial^* E$. Summarising, up to $\H^{N-1}$-negligible subsets
\[
T^+_{\rm ltr} \subseteq \partial Q_\eps\times (\sigma^-,\sigma^-+\delta) \cup \tau\big(\partial^* E\cap (\partial Q_\eps\times (\sigma^-+\delta,\sigma^++\delta))\big)\,,
\]
where $\tau$ is defined in~(\ref{hereistau}). As a consequence, also recalling~(\ref{sectN-2small}) we have
\begin{equation}\label{piece3}\begin{split}
\int_{T^+_{\rm ltr}} h(y,\nu_F(y))&\,d\H^{N-1}(y)\\
&\leq M \Big(\H^{N-1}\big(\partial Q_\eps \times (\sigma^-,\sigma^-+\delta)\big)+ \delta \H^{N-2}\big(\partial^* E \cap (\partial Q_\eps\times(-a/2,a/2))\big)\Big)\\
&\leq M\Big( 2\delta (N-1)(a\eps^\gamma)^{N-2}+N \delta 2^{N+1} (a\eps^\gamma)^{N-2}\Big)
\leq 2^{N+2} NM\delta(a\eps^\gamma)^{N-2}\,.
\end{split}\end{equation}
To conclude, we have to consider $T^\pm_{\rm int}$. Since the section $\sigma^-$ has been chosen in such a way that Vol'pert Theorem is true, and in the cylinder $\C^-=Q_\eps\times (\sigma^-,\sigma^-+\delta)$ the set $F$ is simply $(E^{\sigma^-}\cap Q_\eps)\times (\sigma^-,\sigma^-+\delta)$, we readily obtain
\[
\partial^* F\cap \C^- = \big(\partial^*\big(E^{\sigma^-}\big) \cap Q_\eps \big)\times (\sigma^-,\sigma^-+\delta)\,.
\]
Instead, by construction, $\partial^* F\cap (Q_\eps\times (\sigma^-+\delta,\sigma^++\delta))$ is nothing else than a vertical translation of height $\delta$ of the set $\partial^* E\cap (Q_\eps\times (\sigma^-,\sigma^+))=T^-_{\rm int}$. In particular, if $(x',x_N)\in T^-_{\rm int}$ then $\nu_F(x',x_N+\delta)=\nu_E(x',x_N)$. Let us call for brevity $\xi$ the vertical translation of a height $\delta$, i.e. $\xi(x',x_N)=(x',x_N+\delta)$, so that for $y\in T^-_{\rm int}$ one has $\nu_F(\xi(y))=\nu_E(y)$. By the local $\alpha$-H\"older assumption on $h$ in the spatial variable, we have a constant $C_1$ such that $|h(\xi(y),\nu)-h(y,\nu)|\leq C_1 \delta^\alpha$ for every $y\in Q^N,\, \nu\in\S^{N-1}$. Consequently, by~(\ref{eq:sigma-}) and~(\ref{eq:10}) we can estimate
\begin{equation}\label{piece4}\begin{split}
\int_{T^+_{\rm int}} h(y,\nu_F(y))\, d\H^{N-1}(y)
&-\int_{T^-_{\rm int}} h(y,\nu_E(y))\, d\H^{N-1}(y)\\
&\leq M\H^{N-1}\big(\partial^* F\cap \C^-\big) + \int_{T^-_{\rm int}} h(\xi(y),\nu_E(y))-h(y,\nu_E(y))\,d\H^{N-1}(y)\\
&\leq M\delta \H^{N-2}\Big(\partial^*\big(E^{\sigma^-}\big)\cap Q_\eps\Big) + C_1\delta^\alpha \H^{N-1} (T^-_{\rm int})\\
&\leq 3\cdot 2^{N+3} M\delta \rho a^{N-2}\eps^{\gamma (N-1)} + C_1\delta^\alpha (a\eps^\gamma)^{N-1} (1+2^{N+5}\rho)\,.
\end{split}\end{equation}
Putting together~(\ref{piece1}), (\ref{piece2}), (\ref{piece3}) and~(\ref{piece4}), and recalling that $\delta\leq \bar\delta$ together with the definition~(\ref{precisebardelta}) of $\bar\delta$, if $\bar\eps$ is small enough we finally derive
\begin{equation}\label{pregamma}
P(F)-P(E) \leq C_2 \Big(\bar\delta(\eps^\gamma)^{N-2} + \bar\delta^\alpha (\eps^\gamma)^{N-1}\Big)
\leq C_3 \Big(\eps^{1-\gamma} + \eps^{\alpha + \gamma(N-1)(1-\alpha)}\Big)
\end{equation}
for two constants $C_2,\,C_3$ only depending on $N,\,M$ and $a$, so actually on $N,\, M$ and $E$.

\step{VIII}{Optimal choice of $\gamma$ and definition of $\beta$.}
The estimate~(\ref{pregamma}) proved in the above step holds for a generic $\gamma$. Keeping in mind that $\gamma$ could be chosen depending only on $N$ and $\alpha$, and that the construction requires $0\leq \gamma<1/(N-1)$, we here make an optimal choice. Notice that $\gamma\mapsto 1-\gamma$ is strictly decreasing while $\gamma\mapsto \alpha+\gamma(N-1)(1-\alpha)$ is increasing. As $1-\gamma\geq \alpha+\gamma(N-1)(1-\alpha)$ for $\gamma=0$ while $1-\gamma< \alpha+\gamma(N-1)(1-\alpha)$ for $\gamma=1/(N-1)$, the best choice is the unique value $\gamma$ such that $1-\gamma=\alpha+\gamma(N-1)(1-\alpha)$, i.e.
\begin{equation}\label{tobecited}
\gamma = \frac{1-\alpha}{\alpha+N(1-\alpha)}\,.
\end{equation}
With this choice, the estimate~(\ref{pregamma}) becomes $P(F)\leq P(E) +C\eps^\beta$ where $\beta=\beta(N, \alpha)$ is given by~(\ref{defbeta}) and $C$ only depends on $M,\, N,\, E,\, \alpha$ and the local H\"older constant of $h$.

\step{IX}{The case of a continuous function $h$.}
In this step we consider in more details the case $\alpha=0$, in which ``locally $\alpha$-H\"older'' is a synonym of ``locally bounded''. In this case, $\gamma=1/N$ and $\beta=(N-1)/N$, so in the previous steps for a small $0<\eps<\bar\eps$ we already found a set $F$ such that $|F|=|E|+\eps$ and $P(F)\leq P(E) + C\eps^{\frac{N-1}N}$. In view of applications, for instance Theorem~\ref{thm:boundedness}, it is important to be able to choose the constant $C$ arbitrarily small. More precisely, for any given constant $\kappa>0$, one would be interested to have $C\leq\kappa$ up to choosing $\bar\eps$ accordingly. In particular, $\bar\eps$ should also depend on $\kappa$ and clearly if $\kappa$ becomes very small, so must $\bar\eps$. It is simple to see that this improvement is false for a generic locally bounded $h$, see for instance~\cite{CP}. Here we show that such an improvement is possible if $h$ is continuous in the spatial variable.\par

Let us assume that $\alpha=0$ and that $h$ is continuous in the spatial variable. We only need a slight yet fundamental modification of our argument. More precisely, we fix a large constant $L=L(M,\,N,\, E,\, \kappa)$, to be specified later on, and we possibly decrease, depending on $\kappa$ and on $L$, the value of $\bar\eps$ found in the previous steps: we let $\overline{\eps'}=\bar\eps/L$. For every $0<\eps<\bar\eps$, we call $\eps'=\eps/L$, so that $\eps'$ can be any number in the interval $\big(0,\overline{\eps'}\big)$. We aim to find some set $F$, which equals $E$ outside $Q^N$, such that
\begin{align}\label{wantnow}
|F|=|E|+ \eps'\,, && P(F) \leq P(E) + C \eps'^{\frac{N-1}N}\,.
\end{align}
Correspondingly, we also modify  the definition of $\bar\delta$, which in place of~(\ref{precisebardelta}) is now
\begin{equation}\label{precisebardeltanow}
\bar \delta = \frac{2M\eps}{L(a\eps^\gamma)^{N-1}}= \frac{2M\eps^{1/N}}{La^{N-1}}\,.
\end{equation}
As already done immediately after~(\ref{precisebardelta}), we recall that the construction only makes sense with $\bar\delta\ll a$, and this is again true as soon as $\bar\eps$ is small enough. Up to these modifications, we keep everything unchanged in the first $5$ steps. In Step~VI, the volume estimates are again true, and in particular~(\ref{toobig}) and~(\ref{toosmall}) now read
\begin{gather*}
|F_{\bar\delta}| - |E|\geq
\bar\delta\, (a\eps^\gamma)^{N-1}\Big(\frac {1- 2^{N+5}\rho}M - 2^{N+7}M \rho\Big)
> \frac\eps L = \eps' \,, \\
|F_{\bar\delta/(4M^2)}| - |E| \leq \frac{\bar\delta}{4M^2} \,(a\eps^\gamma)^{N-1}\big(M +2^{N+6} M \rho\big)
=\frac\eps{2LM} \big(M +2^{N+6} M \rho\big)<\frac\eps L =\eps'\,,
\end{gather*}
as soon as $\rho$ is small enough and depending only on $N$ and $M$. Again by continuity, we have then the existence of a constant $\bar\delta/(4M^2)<\delta<\bar\delta$ such that, calling $F=F_\delta$, the equality $|F|-|E|=\eps'$ holds. We have then to bound the perimeter of $F$, and this will be done by using the estimates of Step~VII with only a single modification.\par

Exactly in Step~VII, we divide $\partial^* F \cap \overline\C$ and $\partial^* E \cap \overline\C$ in the parts $T^\pm_{\rm top}$, $T^\pm_{\rm btm}$, $T^\pm_{\rm ltr}$ and $T^\pm_{\rm int}$ by~(\ref{ult1}) and~(\ref{ult2}). Recall that $\H^{N-1}\big(T^+_{\rm btm} \cup T^-_{\rm btm}\big) =0$ by~(\ref{piece1}), while by~(\ref{piece2}) and~(\ref{piece3}), also recalling~(\ref{precisebardeltanow}) and that $\gamma=1/N$, we have
\begin{equation}\label{piece23new}
\int_{T^+_{\rm top}\cup T^+_{\rm ltr}} h(y,\nu_F(y))\,d\H^{N-1}(y) \leq \frac{2^{N+7}NM^2}{aL} \big( \rho  \eps
+ \eps^{\frac{N-1}N}\big) = C_4 \Big(\rho \eps' + \frac 1{L^{1/N}}\,\eps'^{\frac{N-1}N}\Big)< \frac \kappa 3 \,\eps'^{\frac{N-1}N}  \,,
\end{equation}
where $C_4$ is a constant only depending on $M\,, N$ and $a$, so again on $M\,, N$ and $E$, and the last step is true as soon as $L$ is large enough and $\bar\eps$ is small enough, both depending on $M,\,N,\, E$ and $\kappa$.\par

To conclude, we need to evaluate the perimeter contribution of $T^\pm_{\rm int}$. We will argue similarly as what already done in the estimate~(\ref{piece4}). The only difference is that this time we do not estimate
\[
h(\xi(y),\nu_E(y))-h(y,\nu_E(y))\leq C_1\delta^\alpha=C_1
\]
by using the local boundedness, rather we use the continuity of $h$ in the spatial variable. This implies uniform continuity when the spatial variable is inside the cube $Q^N$. More precisely, we call $\omega$ the continuity modulus of $h$ in the spatial variable inside $Q^N$, that is
\[
\omega(d) = \sup \Big\{|h(y,\nu)-h(z,\nu)|:\, y,\,z\in Q^N,\, |y-z|\leq d,\, \nu\in\S^{N-1}\Big\}\,.
\]
Then, calling $\xi$, as in Step~VII, the vertical translation of height $\delta$, and recalling that $\delta<\bar\delta$ and~(\ref{precisebardeltanow}), we have
\[
\big| h(\xi(y),\nu_E(y)) - h(y,\nu_E(y))\big| \leq \omega(\bar\delta) = \omega\bigg( \frac{2M\eps^{1/N}}{La^{N-1}}\bigg)\,.
\]
Putting this estimate inside~(\ref{piece4}), we now get
\begin{equation}\label{piece4new}\begin{split}
\int_{T^+_{\rm int}} h(y,\nu_F(y))\, d\H^{N-1}(y)&-\int_{T^-_{\rm int}} h(y,\nu_E(y))\, d\H^{N-1}(y)\\
&\leq 3\cdot 2^{N+3} M\bar\delta \rho a^{N-2}\eps^{\gamma (N-1)} + \omega(\bar\delta) (a\eps^\gamma)^{N-1} (1+2^{N+5}\rho)\\
&\leq C_5 \eps'  + 2\omega(\bar\delta) a^{N-1} L^{\frac{N-1}N} \,\eps'^{\frac{N-1}N}
< \frac \kappa 3 \eps'^{\frac{N-1}N}  + 2\omega(\bar\delta) a^{N-1} L^{\frac{N-1}N} \,\eps'^{\frac{N-1}N}\,,
\end{split}\end{equation}
where as usual $C_5=C_5(M,\,N,\, a)$ and the last estimate is true for $\bar\eps$ small enough, as usual depending on $M,\,N,\, E$ and $\kappa$.\par

We can finally conclude. Indeed, the first thing to do is to fix $L$, depending on $M,\,N,\, E$ and $\kappa$, so large that~(\ref{piece23new}) holds true. Once $L$ has been fixed, we fix $\bar\eps$ as small as desired, depending on $M,\,N,\, E,\, \kappa$ and $L$. Since $h$ is continuous in the spatial variable, the continuity modulus $\omega(d)$ goes to $0$ as $d$ goes to $0$. Hence, if $\bar\delta$ is small enough, we have that
\[
2\omega(\bar\delta) a^{N-1} L^{\frac{N-1}N} <\frac \kappa 3\,.
\]
Recalling the definition~(\ref{precisebardeltanow}) of $\bar\delta$, we deduce that the above inequality is true as soon as $\bar\eps$ is small enough, depending on $M,\,N,\,L$ and $a$. Inserting this last inequality in~(\ref{piece4new}) and recalling~(\ref{piece23new}), we obtain $P(F)\leq P(E) + \kappa \eps'^{\frac{N-1}N}$. Summarising, in the case when $\alpha=0$ but $h$ is continuous in the spatial variable, for every $\kappa>0$ we have found some $\overline{\eps'}>0$ depending on $M,\,N,\,E,\, \kappa$ such that for every $0<\eps'<\overline{\eps'}$ there is a set $F$ with $E\Delta F\comp Q^N$ satisfying~(\ref{wantnow}). The $\eps-\eps^\beta$ property with constant $\kappa$ has then been proved, at least for positive values of $\eps'$.

\step{X}{Conclusion (i.e., the case $\eps<0$).}
In this last step, we finally conclude the proof. We are left to consider the case $\eps<0$, which is in fact a simple consequence of what we have proved so far. Summarising, we have taken a set $E\subseteq\R^N$ of locally finite perimeter, a ball $B$ with $\H^{N-1}(\partial^* E\cap B)>0$, and a point $x\in \partial^* E \cap B$, and we found two positive constants $C$ and $\bar\eps$ with the property that, for every $0<\eps<\bar\eps$, there exists a set $F$ with $F\Delta E\comp B$ and satisfying~(\ref{epsepsbeta}), that is
\begin{align*}
F\Delta E \comp B\,, && |F| - |E| = \eps\,, && P(F)-P(E) \leq C\eps^\beta\,.
\end{align*}
We can apply this result to the set $\widehat E=B\setminus E$, which is also a set of locally finite perimeter and whose reduced boundary also have intersection with $B$ of strictly positive $\H^{N-1}$-measure, since $\partial^* \widehat E\cap B=\partial^* E\cap B$. Then, we have two positive constants $\widehat C$ and $\widehat{\bar\eps}$ such that for every $0<\eps<\widehat{\bar\eps}$, there exists a set $\widehat F$ such that
\begin{align*}
\widehat F\Delta \widehat E \comp B\,, && |\widehat F| - |\widehat E| = \eps\,, && P(\widehat F)-P(\widehat E) \leq \widehat C\eps^\beta\,.
\end{align*}
Defining then $F=\big(E\setminus B\big)\cup \big(B\setminus \widehat F\big)$, we clearly have
\begin{align*}
F\Delta E \comp B\,, && |F| - |E| = -\eps\,, && P(F)-P(E) \leq C'|-\eps|^\beta\,.
\end{align*}
In other words, we automatically have the validity of~(\ref{epsepsbeta}) also for negative values of $\eps$, up to possibly decreasing the value of $\bar\eps$ (resp., increasing the value of $C$), in case $\widehat{\bar\eps}$ is smaller (resp., $\widehat C$ is larger). The proof is then concluded.
\end{proof}

\section{Boundedness and regularity of isoperimetric sets}\label{sec:applications}

In this last section we give the proofs of Theorem~\ref{thm:boundedness} and of Theorem~\ref{thm:regularity}, which respectively deal with the boundedness and the regularity of isoperimetric sets.

\subsection{Boundedness}

Let us start with the proof of the boundedness of isoperimetric sets. We underline that the assumptions of the theorem, namely, the boundedness of the densities and the continuity of $h$, are both necessary. Indeed, even in the special case of single density, it is possible to find unbounded isoperimetric sets when either the boundedness or the continuity assumption is dropped (see~\cite{MP,CP}).

\begin{proof}[Proof of Theorem~\ref{thm:boundedness}]
Let $E$ be an isoperimetric set as in the claim. In particular, we can find a ball $B$ and some $\bar\eps>0$ such that for every $-\bar\eps<\eps<\bar\eps$ there exists a set $F$ satisfying~(\ref{epsepsbeta}) with constant
\begin{equation}\label{eq:choiceC}
C = \frac{N\omega_N^{\frac 1N}}{2M^\frac{2N-1}N}\,.
\end{equation}
Let us fix $R_0\gg 1$ such that, calling $B_{R_0}=\{x\in\R^N:\, |x|<R_0\}$, one has $B\subseteq B_{R_0}$ and $|E\setminus B_{R_0}|<\bar\eps$. For every $R>R_0$, let us define
\[
\varphi(R) = |E\setminus B_R| \,,
\]
which is a bounded and locally Lipschitz decreasing function, thus is particular in $W^{1,1}(\R)$. For every $R>R_0$, we can take a set $F$ satisfying~(\ref{epsepsbeta}) with $\eps=\varphi(R)$, so that $|F\cap B_R|=|E|$. Since $E$ is an isoperimetric set, keeping in mind~(\ref{eq:bounds}) we can evaluate
\[\begin{split}
P(E)&\leq P(F\cap B_R) = P(F)+P(E\cap B_R) -P(E)\\
&\leq P(F) - P(E\setminus B_R) + \int_{\partial B_R\cap E} h(x,x/|x|)+h(x,-x/|x|) \,d\H^{N-1}(x)\\
&\leq P(E) + C\varphi(R)^{\frac{N-1}N} -P(E\setminus B_R) + 2M \H^{N-1}(\partial B_R\cap E)\\
&\leq P(E) + C\varphi(R)^{\frac{N-1}N} -P(E\setminus B_R) - 2M^2 \varphi'(R)\,,
\end{split}\]
recalling that $\varphi'(R) = - \int_{\partial B_R\cap E} f(x)\,d\,\H^{N-1}(x)$ holds for $\H^{N-1}$-almost every $R$.

By~(\ref{eq:bounds}) and by the Euclidean isoperimetric inequality we have
\[
P(E\setminus B_R) \geq \frac 1M \, P_{\rm Eucl}(E\setminus B_R) \geq \frac {N\omega_N^{1/N}}M \, |E\setminus B_R|_{\rm Eucl}^{\frac{N-1}N}
\geq \frac {N\omega_N^{1/N}}{M^{\frac{2N-1}N}} \, \varphi(R)^{\frac{N-1}N}=2C\varphi(R)^{\frac{N-1}N}\,.
\]
Pairing it with the previous estimate we get
\[
|\varphi'(R)| \geq \frac C{2M^2} \,\varphi(R)^{\frac{N-1}N}\,.
\]
Since $\frac{N-1}N<1$, this inequality implies the existence of some $R_1>R_0$ such that $\varphi(R_1)=0$, that is, the set $E$ is bounded.
\end{proof}

\subsection{Regularity}

To present the proof of the regularity of isoperimetric sets, we first need to recall some classical definitions and results.

\begin{defin}\label{defrpp}
We say that a set $E\subseteq \R^N$ of locally finite perimeter is \emph{quasi-minimal} if there exists a constant $K>0$ such that, for every ball $B_r(x)$, one has
\begin{equation}\label{quasiminimality}
P_{\rm Eucl}(E, B_r(x)) \leq Kr^{N-1}\,.
\end{equation}
We say it is \emph{$\omega$-minimal}, for some continuous and increasing function $\omega:\R^+ \to \R^+$ with $\omega(0)=0$, if for every ball $B_r(x)$ and every set $H\subseteq\R^N$ with $H\Delta E \comp B_r(x)$ one has
\[
P_{\rm Eucl}(E, B_r(x)) \leq P_{\rm Eucl}(H, B_r(x)) +\omega(r)r^{N-1}\,.
\]
We say that a set $E\subseteq\R^N$ is \emph{porous} if there exists a constant $\delta>0$ such that, for every $x\in\partial E$ and every $r>0$ small enough (possibly depending on $x$), there are two balls $B_1,\,B_2\subseteq B_r(x)$, both with radius $\delta r$, such that $B_1\subseteq E$ and $B_2\subseteq \R^N\setminus E$.
\end{defin}

Putting together well-known results, see~\cite{DS, GG, KKLS, Tam}, we have the following.

\begin{thm}\label{thm:stdreg}
Let $E$ be a set of locally finite perimeter. If it is locally quasi-minimal, then it is porous and $\partial E = \partial^* E$ up to $\H^{N-1}$-negligible sets. Additionally, if it is locally $\omega$-minimal with $\omega(r) = Cr^{2\eta}$ for some $\eta \in (0,1/2]$, then $\partial E$ is ${\rm C}^{1, \eta}$.
\end{thm}

We are now ready to prove Theorem~\ref{thm:regularity}.

\begin{proof}[Proof of Theorem~\ref{thm:regularity}]
Let $E\subseteq\R^N$ be an isoperimetric set. Thanks to Theorem~\ref{thm:stdreg}, we only need to check that $E$ is locally quasi-minimal whenever $f$ and $h$ are locally bounded, and that if in addition $h=h(x)$ is $\alpha$-H\"older for some $0<\alpha\leq 1$ then $E$ is also locally $\omega$-minimal with $\omega(r)=Cr^{2\eta}$, being $\eta$ as in~(\ref{defeta}).\par

We fix a generic ball $B\subseteq\R^N$, and we only need to consider balls $B_r(x)\comp B$. Since $f$ and $h$ are locally bounded and l.s.c., there exists some constant $M>0$ such that
\begin{align*}
\frac 1M \leq f(x) \leq M\,, && \frac 1M \leq h(x, \nu) \leq M\,,&& \forall\,x\in B,\,\nu\in\S^{N-1}\,.
\end{align*}
Let us start to consider the quasi-minimality. First of all, by Theorem~\ref{thm:e-ebeta} we know that the $\eps-\eps^{\frac{N-1}N}$ property holds for $E$. Then, take two disjoint balls $B_1$ and $B_2$, both intersecting $\partial^* E$ in a set of positive $\H^{N-1}$-measure, and let $C_1,\, C_2$ and $\bar\eps_1,\,\bar\eps_2$ be the corresponding constants according to~(\ref{epsepsbeta}). Let also $C=\max\{C_1,\,C_2\}$ and $\bar\eps=\min\{\bar\eps_1,\,\bar\eps_2\}$, and let
\[
\bar r = \min \bigg\{ \bigg(\frac{\bar\eps}{M\omega_N}\bigg)^{1/N},\, {\rm dist}(B_1,\,B_2)\bigg\}\,.
\]
Let now $B_r(x)\comp B$ be any ball. If $r<\bar r$, then by definition of $\bar r$ we have that
\begin{equation}\label{epqu}
\eps:=|B_r(x)\cap E|< M \omega_N r^N<\bar\eps\,,
\end{equation}
and that $B_r(x)$ cannot intersect both $B_1$ and $B_2$. Thus, without loss of generality we may assume that $B_r(x)\cap B_1=\emptyset$. As a consequence, by the $\eps-\eps^{\frac{N-1}N}$ property we can find a set $F\subseteq\R^N$ such that $F\Delta E\comp B_1$, $|F|=|E|+\eps$ and $P(F)\leq P(E) + C\eps^{\frac{N-1}N}$. Calling then $G=F\setminus B_r(x)$, we have that $|G|=|E|$, and by the minimality of $E$, (\ref{eq:bounds}) and~(\ref{epqu}) we have
\[\begin{split}
P(E)&\leq P(G) \leq P(F) - P(E,B_r(x))+ N\omega_NM r^{N-1}
\leq P(E)+ C \eps^{\frac{N-1}N} - P(E,B_r(x))+ N\omega_NM \,r^{N-1} \\
&\leq P(E) - P(E,B_r(x))+ N\omega_NM \,r^{N-1} + C M^{\frac{N-1}N}\omega_N^{\frac{N-1}N} r^{N-1}\,.
\end{split}\]
Hence,~(\ref{quasiminimality}) holds true for $B_r(x)$ with the choice
\[
K=N\omega_NM^2 + C M^{1+\frac{N-1}N}\omega_N^{\frac{N-1}N} \,.
\]
If otherwise $r\geq \bar r$, we clearly have
\[
P_{\rm Eucl}(E,B_r(x))\leq M P(E,B_r(x)) \leq MP(E) \leq \frac{MP(E)}{\bar r^{N-1}}\, r^{N-1}\,,
\]
so that~(\ref{quasiminimality}) again holds true for $B_r(x)$. The local quasi-minimality of $E$ is then proved.\par
Let us then assume that $h$ only depends on the spatial variable and that it is $\alpha$-H\"older, so that by Theorem~\ref{thm:e-ebeta} we have the validity of the $\eps-\eps^\beta$ property with $\beta$ given by~(\ref{defbeta}). To conclude the proof, we have to check the $\omega$-minimality of $E$ for balls $B_r(x)\comp B$ with $\omega(r)=\overline Cr^{2\eta}$, being $\eta$ as in~(\ref{defeta}) and being $\overline C$ a suitable constant. Up to increasing the constant $\overline C$, we can restrict ourselves to consider only radii $r<\bar r$, being $\bar r$ as before. Let then $B_r(x)\subseteq B$ be any ball, and $H\subseteq\R^N$ a set such that $H\Delta E\comp B_r(x)$. Moreover, let us call $\eps=|E|-|H|$, which satisfies $|\eps|< \bar\eps$ since $r<\bar r$. Arguing exactly as in the first part of the proof, we can find a set $F\subseteq\R^N$ such that $F\Delta E$ is a positive distance apart from $B_r(x)$, that $|F|=|E|+\eps$, and that $P(F)\leq P(E) + C |\eps|^\beta$. Calling $G=\big(F \setminus B_r(x)\big)\cup \big(H\cap B_r(x)\big)$, by construction we have that $|G|=|E|$, so by the isoperimetric property of $E$ we can evaluate
\[
P(E)\leq P(G)= P(F) +P(H,B_r(x))-P(E,B_r(x))\leq P(E) +C|\eps|^\beta +P(H,B_r(x))-P(E,B_r(x))\,.
\]
Calling then $m=\max \{ h(x),\, x\in B_r(x)\}\geq 1/M$, by the $\alpha$-H\"older property of $h$ and by~(\ref{epqu}) we have
\[\begin{split}
(m-C'r^\alpha) P_{\rm Eucl}(E,B_r(x))&\leq P(E,B_r(x)) \leq P(H,B_r(x)) +C|\eps|^\beta \leq m P_{\rm Eucl}(H,B_r(x)) + C|\eps|^\beta\\
&\leq mP_{\rm Eucl}(H,B_r(x)) + C (M\omega_N r^N)^\beta\,.
\end{split}\]
By the first part of the proof $E$ is quasi-minimal, thus
\[
P_{\rm Eucl}(E,B_r(x)) \leq P_{\rm Eucl}(H,B_r(x)) + CM^{1+\beta}\omega_N^\beta r^{N\beta} + C' M K r^{\alpha+N-1}\leq \overline C r^{N\beta}\,,
\]
where the last inequality comes from the fact that $N\beta \leq \alpha+N-1$, as one readily obtains from~(\ref{defbeta}). We established the $\omega$-minimality of $E$ with $\omega(r)=\overline C r^{N\beta-(N-1)}$, which completes the proof since we have $N\beta-(N-1)=2\eta$.
\end{proof}

\bibliographystyle{plainurl}

\bibliography{e-ebeta-doubled}

\end{document}